\documentclass[preprint,12pt]{elsarticle}
\usepackage{amsfonts}
\usepackage{amsthm}
\usepackage{amsmath}
\usepackage{amsfonts}
\usepackage{latexsym}
\usepackage{amssymb}
\usepackage[latin1]{inputenc}
\usepackage{verbatim}
\usepackage[nodots]{numcompress}
\newtheorem{theorem}{Theorem}[section]
\newtheorem{lemma}[theorem]{Lemma}
\newtheorem{proposition}[theorem]{Proposition}
\newtheorem{corollary}[theorem]{Corollary}

\theoremstyle{definition}
\newtheorem{definition}{Definition}[section]

\newtheorem{notation}{Notation}[section]

\theoremstyle{remark}

\usepackage[active]{srcltx}

\newcommand\mn{{\rm M}_n}

\newcommand\mk{{\rm M}_k}

\newcommand\un{{\rm U}_n}                %%% unitary group
\def\cn{{\mathbb C}^n}
\def\ct{{\mathbb C}\,[t]}
\def\ut{{\rm UpperToepl}_n}
\def\alg{ {\rm Alg}\,}
\def\Hom{{\rm Hom}\,}

\journal{Linear Algebra and its Applications}

\begin{document}

%%%%%%%%%%%% Frontmatter %%%%%%%%%%%%%%
\begin{frontmatter}

\title{A Complete Unitary Similarity Invariant for Unicellular Matrices}
\author{Douglas Farenick}
\address{Department of Mathematics and Statistics,
University of Regina,
Regina, Saskatchewan S4S 0A2, Canada}

\author{Tatiana
G.~Gerasimova}
\address{The Faculty of Mechanics and Mathematics,
Kiev National Taras Shevchenko University, Volodymyrska St, 64,
Kiev-33, 01033, Ukraine}

\author{Nadya Shvai}
\address{The Faculty of Mechanics and Mathematics,
Kiev National Taras Shevchenko University, Volodymyrska St, 64,
Kiev-33, 01033, Ukraine}

\begin{abstract}
We present necessary and sufficient conditions for an $n\times n$ complex matrix $B$ to be
unitarily similar to a fixed unicellular (\emph{i.e.}, indecomposable by similarity) $n\times n$
complex matrix $A$.
\end{abstract}

\end{frontmatter}
%%%%%%%%%%%%%% End of Frontmatter %%%%%%%%%%%%%%%

\section{Introduction}

A fundamental problem in matrix analysis is the \emph{unitary similarity
problem} \cite{arveson1970,shapiro1991}:  Under what necessary and sufficient conditions are two $n\times n$ complex
matrices unitarily similar?
A classical and  purely algebraic solution to this problem due to
Specht \cite{Kaplansky-linear-algebra-book,specht1940}: two $n\times n$ complex
matrices $A$ and $B$ are unitarily
similar if and only if
\begin{equation}\label{e:Specht}
\mbox{Trace}\,\omega(A,A^*)\;=\;\mbox{Trace}\,\omega(B,B^*)\,,
\end{equation}
for every word $\omega$ in two noncommuting variables $x$ and $y$.

In many applications, the data one has about a particular matrix are not based on the trace of the matrix,
but rather on some other analytical information: the spectrum or pseudospectrum, the numerical range
or polynomial numerical hull, the singular values, a unitarily invariant norm, and so forth.
Our concern in the present paper is with a solution to the unitary similarity problem
that is based on a particular choice of unitarily invariant norm.

Let $\mn$ be the space of all $n\times n$ complex matrices; we denote the unitary group by $\un$.
Two matrices $A,B\in\mn$ are unitarily similar, which we express by $A\sim B$,
if there is a $U\in \un$ such that $B=U^*AU$. The norm under study is defined by
\begin{equation}\label{norm}
\|A\|\,=\,\sqrt{\mbox{\rm spr}\,(A^*A)}\,,
\end{equation}
where $\mbox{spr}\,X$ is the spectral radius of $X\in \mn$. The norm \eqref{norm}
has the property that $\|U^*AU\|=\|A\|$, for all $A\in\mn$ and $U\in\un$, and it
coincides with the largest singular value of $A$. Moreover, if $A\in\mn$ is considered as
a linear transformation on the complex inner product space $\mathbb C^n$ with respect
to the standard inner product $\langle\xi,\eta\rangle=\eta^*\xi$, for $\xi,\eta\in\mathbb C^n$, then
\[
\|A\|\;=\;\max_{\langle\xi,\xi\rangle=\langle\eta,\eta\rangle=1}\,|\langle A\xi,\eta\rangle|\,.
\]

Let $\ct$ denote the ring of polynomials with complex coefficients. If $A\sim B$, then necessarily
$\|f(A)\|=\|f(B)\|$ for all $f\in\ct$. Conversely, if $A,B\in\mn$ are such that
$\|f(A)\|=\|f(B)\|$ for all $f\in\ct$, then $A$ and $B$ yield to the same matrix analysis:

(i)  $A$ and $B$ have the same spectrum;

(ii)  $A-zI$ and $B-zI$ have the same condition numbers, for all nonspectral $z$ in the complex plane;

(iii) $A$ and $B$ have the same polynomial numerical hulls and, in particular, the same numerical range;

(iv) $A$ and $B$ have the same spectral set;

(v) $A$ and $B$ have the same pseudospectrum.

Our first objective is to determine cases in which the condition $\|f(A)\|=\|f(B)\|$ for all $f\in\ct$ is also sufficient
for $A\sim B$. In general it will not be so, for if one takes any two nonzero projections (selfadjoint idempotents) $P$
and $Q$, then one has $ \|f(P)\|=\|f(Q)\|$ for all $f\in\ct$, independent
of the ranks of $P$ and $Q$. Therefore, for questions concerning unitary similarity,
the hypothesis  $\|f(A)\|=\|f(B)\|$ for all $f\in\ct$ is relevant only for the
analysis of nonnormal matrices.

\begin{definition}
A matrix $A\in\mn$ is said to be \emph{unicellular} if $A$ is not similar to a matrix $B\in\mn$ of the form
$B=G\oplus H$, for some square matrices $G$ and $H$ of strictly smaller size than $B$.
\end{definition}

Our use of the term unicellular matrix is motivated by the concept of unicellular operator or transformation
in operator theory. If $A\in\mn$ is a unicellular matrix, then $A$ is unicellular in the sense
of \cite[\S9]{Gohberg--Krein-Volterra-book}, \cite[\S 2.5]{Gohberg-Lancaster-Rodman-book}
as a linear transformation on $\mathbb C^n$.
Unicellular matrices are also said to be \emph{indecomposable by similarity}.

In this paper we present two main results. The first, Theorem \ref{toeplitz}, states that the unitary similarity class of any
upper triangular unicellular Toeplitz matrix $R$ is determined by the values of $\|f(R)\|$ for various $f\in\ct$. If one drops the
requirement that $R$ be Toeplitz, yet remain upper triangular and unicellular, then the values of $\|f(R)\|$, for
$f\in\ct$, are insufficient to identify $R$ up to unitary similarity (Proposition \ref{ws}). But with our second main result,
Theorem \ref{main result}, we augment the criterion slightly to obtain necessary and sufficient conditions that classify unicellular
matrices up to unitary similarity (see, also, Proposition \ref{si}).

%%%%%%%%%%%%%%%%%%%%%%%%%%%%%
\section{Upper Triangular Toeplitz Matrices}

\begin{definition} A matrix $R\in\mn$ is an \emph{upper triangular Toeplitz matrix} if
\begin{equation}\label{toeplitz matrix}
R\,=\,\left[ \begin{array}{ccccc} z_0&z_1&a_2&\cdots &z_{n-1} \\
                                                    0 &z_0&z_1&\ddots &\vdots \\
                                                    0 &0&\ddots&\ddots&z_2 \\
                                                    \vdots &&\ddots&\ddots&z_1 \\
                                                   0  &\dots&\dots&0& z_0 \end{array}
       \right]\,,
\end{equation}
for some $z_0,z_1,\dots, z_{n-1}\in\mathbb C$.

The set of all upper triangular Toeplitz matrices $R\in\mn$ is denoted by
$\ut$.
\end{definition}

The main theorem of this section is:

\begin{theorem}\label{toeplitz} Let $R\in\mn$ be an upper triangular Toeplitz matrix \eqref{toeplitz matrix} with $z_1\neq0$.
If $A\in\mn$ is any matrix for which $\|f(A)\|=\|f(R)\|$, for all $f\in\ct$, then $A\sim R$.
\end{theorem}

Before moving to the proof of Theorem \ref{toeplitz}, let us consider
one of its consequences, namely Corollary \ref{alg iso} below, which is of interest in linear-algebraic analysis.
For any $A\in\mn$, the unital algebra ${\rm Alg}\,A$ generated by $A$ is
\[
{\rm Alg}\,A\,=\,\{f(A)\,:\,f\in\ct\}\,.
\]
In particular, $\ut={\rm Alg}\,S$, where
\[
S\;=\;
\left[ \begin{array}{ccccc} 0& 1&0&\cdots &0 \\
                                                    0 &0&1&\ddots &\vdots \\
                                                    0 &0&\ddots&\ddots&0 \\
                                                    \vdots &&\ddots&\ddots&1 \\
                                                   0  &\dots&\dots&0& 0 \end{array}\right]\,.
\]
More generally, if $R\in\ut$ is of the form \eqref{toeplitz matrix} and satisfies $z_1\neq0$, then the range of $R-z_0I$
is clearly $(n-1)$-dimensional and so the kernel of $R-z_0I$ is $1$-dimensional. Thus, there is an invertible
$X\in \mn$ for which $S=X(R-z_0I)X^{-1}$, the Jordan canonical form of $R-z_0I$.
Hence, the abelian algebras $\ut={\rm Alg}\,S$ and ${\rm Alg}\,R$
are isomorphic. Because ${\rm Alg}\,R$ is a subalgebra of $\ut$, they can be isomorphic only if they are equal.
Thus, if $R\in\ut$ satisfies $z_1\neq0$, then $R$ is called a \emph{generator} of $\ut$.
(Consideration of the Jordan form shows that this necessary condition on $z_1$ is also sufficient
for $R\in\ut$ to be a generator of $\ut$, but we do not require this fact.)

 \begin{corollary}\label{alg iso} If $\varrho:\ut\rightarrow\mn$ is a homomorphism such that
$\|\varrho(X)\|=\|X\|$, for every $X\in\ut$, then there is a $U\in\un$ such that
$\varrho$ is given by
$\varrho(X)=U^*XU$.
\end{corollary}

\begin{proof} Choose $R\in\ut$ of the form \eqref{toeplitz matrix} with $z_1\neq0$ and let $A=\varrho(R)$.
Thus, $f(A)=\varrho(f(R))$, for all $f\in\ct$. By hypothesis,
$\|f(A)\|=\|\varrho(f(R))\|=\|f(R)\|$, for all $f\in\ct$; therefore,
Theorem \ref{toeplitz} asserts that $A=U^*RU$ for some $U\in\un$.
Because $R$ generates $\ut$, we conclude that $\varrho(X)=U^*XU$, for every $X\in\ut$.
\end{proof}

 We move now to the proof of Theorem \ref{toeplitz}.

\subsection{Lemmas}

\begin{lemma}\label{lemma-1} If
\begin{equation}\label{Q}
Q\;=\;\left[ \begin{array}{ccccc} 0&1&1&\cdots &1 \\
                                                     &0&1&\ddots &\vdots \\
                                                     &&\ddots&\ddots&\vdots \\
                                                     &&&\ddots&1 \\
                                                     &&&& 0 \end{array}
       \right]\,,
\end{equation}
then $\displaystyle\sum_{k=1}^\infty\,(-1)^{k+1}Q^k\,=\,S$.
\end{lemma}

\begin{proof} Clearly $Q=\displaystyle\sum_{k=1}^\infty\,S^k$. Thus,
$I+Q\;=\;\displaystyle\sum_{j=0}^\infty\,S^j\;=\;(I-S)^{-1}$,
whence $I=(I-S)(I+Q)$. That is, $S=I-(I+Q)^{-1}=\displaystyle\sum_{k=1}^\infty\,(-1)^{k+1}Q^k$.
\end{proof}

\begin{lemma}\label{lemma-2} Let $Q\in\mn$ be given by \eqref{Q}. If
\[
A\;=\; \left[ \begin{array}{ccccc} 0&1&a_{13}&\cdots &a_{1n}\\
                                                     &0&1&\ddots &\vdots \\
                                                     &&\ddots&\ddots&a_{n-2,n} \\
                                                     &&&\ddots&1 \\
                                                     &&&& 0 \end{array}
       \right]
\]
has the property that $\left\| \displaystyle\sum_{k=1}^\infty\,(-1)^{k+1}A^k\right\|\leq 1$, then $A=Q$.
\end{lemma}

\begin{proof}
We proceed by induction on $n$. The base case is $n=3$. In this case,
\[
 \displaystyle\sum_{k=1}^\infty\,(-1)^{k+1}A^k \;=\; \left[ \begin{array}{ccc}0&1&a_{13}-1 \\ 0&0&1 \\0&0&0\end{array}\right]\,.
 \]
The first row of the matrix above has Euclidean length at most $1$, since $A$ has norm at most $1$. Thus,
$a_{13}=1$, implying that $A=Q$. This row condition extends unchanged to the induction step.

Assume now the statement holds in $n$-dimensional space and consider $A$, $Q$, and $S$
as acting on ${\mathbb C}^{n+1}$. Let $\tilde A$, $\tilde Q$, and $\tilde S$
denote the versions of $A$, $Q$, and $S$ that act on $\cn$, and let
$e_1,\dots,e_n$ denote the canonical orthonormal basis vectors in $\cn$. Hence, as a partitioned matrix, $A$ has the form
\[
A\;=\; \left[ \begin{array}{c|c}&\\
                 \tilde A & \eta \\
                 &\\
                 \cline{1-2}\begin{array}{ccc}  0&\cdots&0\end{array} &0 \end{array}
             \right]\,,
\]
where
\[
\eta\,=\,e_n+\sum_{i=1}^{n-1} a_{i,n+1}e_i\,=\,[a_{1,n+1},\cdots,a_{n-1,n+1}]^T\in\cn\,.
\]
Because
\begin{equation}\label{ineq1}
1\geq \left\| \displaystyle\sum_{k=1}^\infty\,(-1)^{k+1}A^k\right\| \geq
\left\| \displaystyle\sum_{k=1}^\infty\,(-1)^{k+1}\tilde A^k\right\|\,,
\end{equation}
the induction hypothesis yields
$\tilde A=\tilde Q$.
Hence, using Lemma \ref{lemma-1}, we obtain
\[
\displaystyle\sum_{k=1}^\infty\,(-1)^{k+1}A^k \;=\;
 \left[ \begin{array}{c|c}&\\
                \tilde S & \displaystyle\sum_{k=1}^\infty\,(-1)^{k+1}\tilde A^{k-1}\eta \\
                 &\\
                 \cline{1-2}\begin{array}{ccc}  0&\cdots&0\end{array} &0 \end{array}
             \right]\,.
  \]
 That is,
 \begin{equation}\label{easy to see}
  \displaystyle\sum_{k=1}^\infty\,(-1)^{k+1}A^k \;=\;
 \left[ \begin{array}{ccccc|c} 0&1&0& \cdots&0&* \\ \vdots&0&1&\ddots&\vdots&\vdots \\ &&\ddots & \ddots &0&\vdots \\ \vdots&&&\ddots&1&* \\
 0&\cdots&&\cdots&0&1
                                  \\
                 \cline{1-6} 0&\cdots&\cdots&\cdots&0 &0 \end{array}
             \right]\,.
\end{equation}
Similar to the case $n=3$, we have from \eqref{ineq1} that
\begin{equation}\label{ineq2}
1\geq \left\|
\displaystyle\sum_{k=1}^\infty\,(-1)^{k+1}A^k\right\|
\end{equation}
for the
matrix  \eqref{easy to see}. But \eqref{ineq2} holds for the matrix \eqref{easy to see} only if
the $i$-th entry in the final column of the matrix \eqref{easy to see} is $0$, for $1\leq i\leq n-1$.
Therefore,  using $\tilde A=\tilde Q$, we have $(\tilde
I-\tilde S)\eta=\lambda e_n$ for some complex number $\lambda$.
Hence,
\[
\eta \;=\;\lambda(\tilde I-\tilde S)^{-1}e_n
\;=\;\lambda(\tilde I+\tilde S+\tilde S^2+\cdots+\tilde
S^{n-2})e_n\;=\;\lambda(e_n+e_{n-1}+\cdots+e_1)\,.
\]
But on the other hand,
\[
\eta\;=\;e_n+\sum_{i=1}^{n-1} a_{i,n+1}e_i\,,
\]
which implies that $\lambda=1$ and $a_{i,n}=1$ for all $1\leq
i\leq n-1$. Therefore, $A=Q$.
\end{proof}

\subsection{Proof of Theorem \ref{toeplitz}}

Assume first that the matrix $R$ in \eqref{toeplitz matrix} has $z_0=0$ and $z_{j}=1$
for $1\leq j\leq (n-1)$; that is, assume that $R=Q$, where $Q$ has the form \eqref{Q}.
Thus, the hypothesis is that $A\in\mn$ satisfies
$\|f(A)\|=\|f(Q)\|$, for all $f\in\ct$.

By the Spectral Radius Formula,
\[
0\,=\,{\rm spr}\,Q\,=\,\displaystyle\lim_{k\rightarrow\infty}\|Q^k\|^{1/k}
\,=\,\displaystyle\lim_{k\rightarrow\infty}\|A^k\|^{1/k}
\,=\,{\rm spr}\,A\,,
\]
which implies that $A$ is nilpotent.
Without loss of generality, $A$ may be assumed to be in upper triangular
form. Furthermore, using a diagonal unitary similarity transformation, the
entries $a_{i,i+1}$ may assumed to be nonnegative, for $1\le i\le n-1$.
Indeed, since
$1=\|Q^{n-1}\|=\|A^{n-1}\|=|a_{12}a_{23}\cdots a_{n-1,n}|$, each $a_{i,i+1}$ is nonzero; thus, we may
assume that $a_{i,i+1}>0$ for all $i$.

The numerical range, or field of values, $W(X)$ of any $X\in\mn$ is given analytically by
\[
W(X) \;=\;\bigcap_{\alpha,\beta\in{\mathbb C}}\{z\in{\mathbb C}\,:\,|\alpha z+\beta|\le\|\alpha X+\beta 1\|\}\,.
\]
Hence, $W(A)=W(Q)$. Let $\Re(X)=\frac{1}{2}(X+X^*)$, for any $X\in\mn$, and
observe that $\frac{1}{2}+\Re(Q)=\frac{1}{2}\,\xi\otimes\xi$, where $\xi=\sum_{i=1}^ne_i\in\cn$ and $\xi\otimes\xi$ denotes the outer product $\xi\xi^*\in\mn$ of
$\xi$ (a column vector) with its conjugate transpose $\xi^*$. Thus,
for every unit vector $\gamma\in\cn$, the real part of $\langle Q\gamma,\gamma\rangle$ satisfies the inequality
\[
\Re\left(\langle
Q\gamma,\gamma\rangle\right)\;\geq\;-\;\frac{1}{2}\,.
\]
Because $A$ and $Q$ have the same numerical range,
$\Re (A)$ has the same property above. Now, if $P_i$ is the
projection of $\cn$ onto $\mbox{Span}\,\{e_i,e_{i+1}\}$, for each
$1\leq i\leq n-1$, then $P_iAP_i$ as a linear transformation on
the range of $P_i$ is given by
\[
\left[ \begin{array}{cc} 0&a_{i,i+1} \\ 0&0 \end{array}\right]\,.
\]
Therefore, the numerical range of $P_iAP_i$ is a disc of radius
$\frac{1}{2}a_{i,i+1}$ centered at the origin. Because
$W(P_iAP_i)\subseteq W(A) \subset\{z\in{\mathbb C}\,|\,\Re (z)\geq
{-1}/{2}\}$, we conclude that each $a_{i,i+1}\leq 1$. However,
under these conditions the equation
$1=\|A^{n-1}\|=a_{12}a_{23}\cdots a_{n-1,n}$ holds only if
$a_{i,i+1}=1$ for all $1\leq i\leq n-1$. Hence, $A$ has the
structure given in the hypothesis of Lemma \ref{lemma-2}.
Moreover, by Lemma \ref{lemma-1},
\[
1\;=\;\|S\|\;=\;\left\|\sum_{k=1}^\infty(-1)^{k+1}Q^k\right\|\;=\; \left\|\sum_{k=1}^\infty(-1)^{k+1}A^k\right\|\,.
\]
Thus, $A$ satisfies all of the hypotheses of Lemma \ref{lemma-2}, yielding $Q=A$.

For the general case, we now suppose that $R\in\ut$ satisfies $z_1\neq 0$ and
$A\in\mn$ is such that $\|f(A)\|=\|f(R)\|$ for every $f\in\ct$. Therefore, the
ideals $\mathfrak J_A$ and $\mathfrak J_R$ coincide, where for a given $X\in\mn$
\[
\mathfrak J_X\,=\,\{p\in\ct\,:\,p(X)=0\}\,.
\]
Because $R$ is a generator
of $\ut$, there is a $g\in\ct$ such that $Q=g(R)$. Let $B=g(A)$. Thus,
$\|h(B)\|=\|h(Q)\|$, for every $h\in\ct$. By what we proved above, this
yields $B=U^*QU$ for some $U\in\un$. As $Q$ generates $\ut$, there is an $q\in\ct$
such that $R=q(Q)$. Hence,
\[
p(t)\,=\,t-q\left(g(t)\right)\,\in\,\mathfrak J_R\,=\,\mathfrak J_A\,.
\]
This implies that
\[
0\,=\,p(A)\,=\,A-q\left(g(A)\right)\,=\,A-q(B)\,=\,A-U^*q(Q)U\,=\,A-U^*RU\,,
\]
which completes the proof.

%%%%%%%%%%%%%%%
\section{Necessary and Sufficient Conditions for Unitary Similarity}\label{S:main thm}

If $A\in\mn$ is unicellular --- say with spectrum $\{\lambda\}$
--- and if $B\in \mn$ is any matrix for which $\|f(A)\|=\|f(B)\|$
for all $f\in\ct$, then $A$ and $B$ are similar, as the condition
implies that $\sigma(B)=\sigma(A)$ and that
$(B-\lambda I)^{n-1}\not=0$. But, unlike the case for generators of
the upper triangular Toeplitz matrices, $A$ and $B$ need not be
unitarily equivalent (Proposition \ref{ws} below). Therefore, one
can have an invertible matrix $Z\in\mn$ with
\[
\|f(A)\|\;=\;\|Zf(A)Z^{-1}\|\,,\;\mbox{for all }f\in\ct\,,
\]
and yet $Z$ can fail to be unitary.

\begin{proposition}\label{ws} If $0<\alpha<\beta$, then the unicellular matrices
\begin{equation}\label{rh}
A\;=\; \left[ \begin{array}{ccc} 0&\alpha&0 \\ 0&0&\beta \\ 0&0&0 \end{array}\right]
\quad\mbox{ and }\quad
A'\;=\;\left[ \begin{array}{ccc} 0&\beta&0 \\ 0&0&\alpha \\ 0&0&0 \end{array}\right]
\end{equation}
satisfy $\|f(A')\|=\|f(A)\|$ for all $f\in\ct$, but $A' \not\sim A$.
\end{proposition}
\begin{proof}
Note that $A'=W^*A^TW$, where $X\mapsto X^T$ denotes the transpose map and
\[
W\;=\; \left[ \begin{array}{ccc} 0&0&1 \\ 0&1&0 \\ 1&0&0
\end{array}\right] \,.
\]
Because the norm is transpose invariant,
$\|f(A')\|=\|f(A^T)\|=\|f(A)^T\|=\|f(A)\|$, for all
$f\in\ct$.
On the other hand, $A \not\sim A'$ by
Littlewood's algorithm \cite{littlewood1953} because $0<\alpha<\beta$.
(One also can verify
directly that the equation $UA'=AU$ is impossible to satisfy with $U\in{\rm U}_3$. Alternatively,
the referee observed that the matrices $A$ and $A'$ fail to satisfy Specht's tracial condition with the word
$\omega(x,y)=xy^2x^2y$; hence, $A \not\sim A'$.)
\end{proof}

\begin{notation} If $1\leq k\leq n$ and $X=[x_{ij}]_{i,j=1}^n\in\mn$, then $X_k=[x_{ij}]_{i,j=1}^k\in\mk$. That is, $X_k$ is
the leading $k\times k$ principal submatrix of $X$.
\end{notation}

The failure of $A$ and $A'$ in \eqref{rh} to be unitarily similar
is explained by the fact that the norms of $f(A_2)$ and $f(A_2')$ do not
always coincide, even though $\|f(A)\|=\|f(A')\|$ for all $f\in\ct$.
This observation motivates our second main result of the present paper.

\begin{theorem}\label{main result} Assume that $A\in\mn$ is an upper triangular matrix such that
\begin{enumerate}
\item[{\rm (a)}] $a_{ii}=a_{kk}$ for all $1\leq i,k\leq n$, and
\item[{\rm (b)}] $a_{i,i+1}\neq0$, for all $1\leq i\leq (n-1)$ {\rm (}that is, the first superdiagonal of $A$ has only nonzero entries{\rm )}.
\end{enumerate}
Then the following statements are equivalent for an upper triangular matrix $A'\in\mn$:
\begin{enumerate}
  \item\label{mr-one} $\|f(A_i)\|=\|f(A_i')\|$, for all $f\in\ct$ and $1\leq i\leq n$;
   \item\label{mr-two} $A'=W^*AW$ for some diagonal unitary matrix $W\in\un$.
\end{enumerate}
\end{theorem}

\begin{proof} We need only prove that first statement implies the second.

There is a diagonal unitary $W\in\un$ such that the entries in the first superdiagonal of
the upper triangular matrix $W^*AW$ are positive; therefore, without loss of generality we
assume that $a_{i,i+1}>0$ for $1\leq i\leq (n-1)$.
As we argued in the proof of Theorem \ref{toeplitz}, the
condition $\|f(A)\|=\|f(A')\|$, for all $f\in\ct$, implies that $A'$ has one point of spectrum, in this case $\lambda=a_{11}$.
Therefore, by scalar translation $X\mapsto X -\lambda I$ we may assume without loss of generality
that $\lambda=0$. That is,
\begin{equation}\label{e:upper t-form}
A\;=\; \left[ \begin{array}{ccccc}  0 & a_{12} & a_{13} & \dots & a_{1n}\\
                       0 & 0      & a_{23} & \dots & a_{2n}\\
                         &        &\ddots  & \ddots& \vdots      \\
                        &        &        &  0    & a_{n-1 n}\\
                      0 &        &        &       &  0
\end{array}\right]\,,
\end{equation}
where $a_{\ell,\ell+1}>0$, for $1\le \ell\le n-1$.

To complete the proof of theorem, it is sufficient to prove
that the entries of $A$ in (\ref{e:upper t-form}) are completely determined from the values of $\|f(A_i)\|$
for $1\le i\le n$ and all $f\in\ct$.

We shall proceed by induction on $n\ge 3$.

Let $n=3$. Thus,
\begin{equation}\label{e: upper t-form-3}
A\;=\; \left[ \begin{array}{ccc} 0 & a_{12} & a_{13}\\
                       0 & 0      & a_{23}\\
                        0 &      0  &0
\end{array}\right]\,.
\end{equation}
The value of $a_{12}$ is determined via the fact that
$\|A_2\|=a_{12}$, and so the value of $a_{23}$ is determined
from the equation $a_{12}a_{23}=\|A^2\|$. Using $f(t)=t$, we have
\[
\|A\|^2\;=\;\frac{1}{2}\left(a_{12}^2+a_{23}^2+|a_{13}|^2+\sqrt{(a_{12}^2+a_{23}^2+|a_{13}|^2)^2-4a_{12}^2a_{23}^2}\right)\,,
\]
which determines the value of $|a_{13}|$. Two similar calculations using the polynomials
$f(t)=t-\frac{1}{a_{12}a_{23}}t^2$  and $g(t)=t-\frac{i}{a_{12}a_{23}}t^2$ determine the values
of $|a_{13}-1|$ and $|a_{13}-i|$. These last two quantities together with the value of $|a_{13}|$
determine the complex number $a_{13}$,
thereby establishing the base case for the induction.

Assume now that the statement holds for all spaces of dimension up to and including
$n-1$; we will show the statement also holds for spaces of dimension $n$.

For convenience, we denote the entries of $A^k$ by
$a_{ij}^{(k)}$. By the inductive hypothesis, the
entries of the submatrix
$A_{n-1}$ of $A$ are uniquely determined by the norms $\|f(A_j)\|$,
for various $f\in\ct$ and $1\le j\le (n-1)$. Therefore, the
only elements left to consider are those in the final column of $A$: $a_{in}$, $1\le i\le (n-1)$.
We shall obtain these entries in an argument that requires $n-1$ steps; each step uses the
conclusion of the previous step.

\vskip 4pt
\textsc{Step 1}. Recall $A^n=0$ and $A^{n-1}\neq 0$.
The elements of $A^{n-1}$ are zero except in the $(1,n)$ position, where we have
\[
\|A^{n-1}\|\;=\;|a_{1n}^{(n-1)}|\;=\;a_{12}a_{23}\dots a_{n-2, n-1}a_{n-1, n} \,.
\]
Hence, $a_{n-1 n}$ is uniquely determined by the norms $\|f(A_j)\|$ for various $f\in\ct$ and $1\le j\le n$.
This means, in addition, all of the entries of $A^{n-1}$ are now determined.

\vskip 4pt
\textsc{Step $i$}. Assume that $3\le i \le (n-1)$ and that we have completed Steps $1$ to $i-1$, giving us the
values of $a_{j,n}$, for $j=n-i-1,\dots,n-1$ and the entries of each $A^{n-j}$, for $j=1,\dots,i-1$.
We aim to show that the value of $a_{n-i,n}$ is determined from the norms of various $f(A_j)$.

For each complex number $z\in{\mathbb C}$, let
$g_z\in\ct$ be given by $g_z(t)=t^{n-i}+\frac{z}{a_{12}q}t^{n-1}$, where $q=a_{2, n}^{(n-2)}$ (as in Step 2).
Thus,
\[
g_z(A)\;=\; \left[ \begin{array}{ccccccc} 0 & \dots  & 0      & a_{1,n-i+1}^{(n-i)} & a_{1,n-i+2}^{(n-i)} & \dots  & a_{1n}^{(n-i)}+z \\
                               &        &        &    0                & a_{2,n-i+2}^{(n-i)} & \dots  & a_{2n}^{(n-i)} \\
                               &        &        &                     & \ddots              & \ddots & \vdots         \\
                               &        &        &                     &                     &    0   &  a_{in}^{(n-i)}\\
                               &        &        &                     &                     &        &       0        \\
                               &        &        &                     &                     &        &  \vdots        \\
                             0 &        &        &                     &                     &        &  0
\end{array}\right]\,.
\]
Observe that $g_z(A)$ is a rank-1 perturbation (by a matrix unit) of $A^{n-i}$: namely,
\[
g_z(A)\,=\,A^{n-i}+zE_{1,n}\,.
\]
Suppose that there is a
complex number $\tilde{a}_{1,n}^{(n-i)}$ such that
\[
\|\tilde{A}^{n-i}+zE_{1,n}\|\,=\,\|A^{n-i}+zE_{1,n}\|\,,\;\mbox{for all }\,z\in{\mathbb C}\,,
\]
where $\tilde{A}^{n-i}$ is the matrix  obtained from $A^{n-i}+zE_{1,n}$ by replacing $a_{1,n}^{(n-i)}$ by $\tilde{a}_{1,n}^{(n-i)}$.
We shall prove that $\tilde{a}_{1,n}^{(n-i)}=a_{1,n}^{(n-i)}$.
Define a function $h:{\mathbb C}\rightarrow {\mathbb R}_+$ by $h(z)=\|A^{n-i}+zE_{1n}\|$ and let $\gamma=\tilde{a}_{1,n}^{(n-i)}-a_{1,n}^{(n-i)}$.
Thus, $h(z)=h(z+\gamma)$, for all $z\in{\mathbb C}$. In particular, $h(0)=h(k\gamma)$, for all positive integers $k$. However, as it is clear that
$|h(z)|\rightarrow\infty$ as $|z|\rightarrow\infty$, the equations $h(0)=h(k\gamma)$, for all positive integers $k$, can hold only if $\gamma=0$.

Thus, we have shown that the $(1,n)$-entry of $A^{n-i}$, namely
$a_{1,n}^{(n-i)}$, is determined uniquely by the
norms of various $f(A_j)$.

Because the first $n-i-1$ entries in the first row of $A^{n-i-1}$
are zero and because the first $n-i-1$ entries of the last column
of $A$ are $a_{1,n}\dots,a_{1,n-i-1}$, we obtain from
$A^{n-i}=A^{n-i-1}A$ that the $(1,n)$-entry of $A^{n-i}$ is given
by
\begin{equation}\label{e:end}
a_{1,n}^{(n-i)} \;=\;
a_{1,n-i}^{(n-i-1)}a_{1,n-i} \,+\,\sum_{k=1}^{i-1} a_{1,n-i+k}^{(n-i-1)}a_{1,n-i+k}\,.
\end{equation}
Because the entries $a_{1,n}^{(n-i)}$, $a_{1,n-i}^{(n-i-1)}$, $a_{1,n-i+k}^{(n-i-1)}$, and $a_{1,n-i+k}$,
for $1\le k\le i-1$, have already been determined from the
norms of various $f(A_j)$ using the induction hypothesis and Steps $1$ to $i-1$,
(\ref{e:end}) implies that the value of
$a_{1,n-i}$ is determined uniquely from the norms of various $f(A_j)$.

This completes the induction and, hence, the proof of the theorem.
\end{proof}

Note that an upper triangular unicellular matrix $A$ satisfies the hypothesis of Theorem \ref{main result},
and so Theorem \ref{main result} solves the unitary similarity problem in the class of unicellular matrices.

%%%%%%%%%%%%%%%%%%%%%%%%%%%
\section{Application}

%\subsection{The Volterra Integral Operator}

In the theory of integral equations, the classical \emph{Volterra operator} $V$ of integration has some remarkably special
properties \cite{Gohberg--Krein-Volterra-book}.
The operator $V$ is defined as follows: for each $f\in {\rm L}^2([0,1])$, let $Vf\in {\rm L}^2([0,1])$ be given by
\[
Vf\,(t)\;=\;2 i\int_t^1\,f(s)\,ds\,,\quad f\in {\rm L}^2([0,1]),\;t\in[0,1]\,.
\]
In the context of our work in this paper, the operator $V$ is unicellular, which in infinite dimensions
is to say that its closed invariant subspaces are
totally ordered by inclusion.

A question raised many years ago by
Arveson \cite[page 218]{arveson1969} asks whether the norms $\|f(V)\|$, for $f\in\ct$, determine the
unitary similarity class of $V$ in the set of irreducible compact operators on ${\rm L}^2([0,1])$. Although this
question remains open, we prove below that given any $\varepsilon>0$ there is a unicellular
piece $A$ of the Volterra operator whose norms $\|f(A)\|$ determine its unitarily similarity class
and such that $A$ is within $\varepsilon$ of $V$ uniformly on ${\rm L}^2([0,1])$.

\begin{proposition} For every $\varepsilon>0$ there is a finite-dimensional subspace
${\rm L}\subset {\rm L}^2([0,1])$ such that, if $P$ denotes the projection onto $\rm L$, then
\begin{enumerate}
\item $PVP_{\vert {\rm L}}$ is a unicellular operator whose unitary similarity orbit,
as an operator on $\rm L$, is completely determined by the norms $\|f(PVP_{\vert {\rm L}})\|$, for $f\in\ct$, and
\item $\|PVP-V\|<\varepsilon$.
\end{enumerate}
\end{proposition}

\begin{proof}
We use an approximation scheme of Davies and Simon \cite{davies--simon2006}, which they employed to
compute the norm of $V$.
For each positive integer $m$, let ${\rm H}_m$ be the Hilbert space spanned by the $m$ orthonormal functions
$\sqrt{m}\chi_{E_j}$, $0\leq j\leq n-1$, where $E_j=[\frac{j}{m},\frac{j+1}{m})$. If $P_m$ is the projection with range ${\rm H}_m$, then
$P_mVP_m$ considered as an operator on ${\rm H}_m$ has a matrix representation
with respect to this orthonormal basis of ${\rm H}_m$ that is given by
\[
P_mVP_m{}_{\vert {\rm H}_m}\;=\;\frac{i}{m}\left(1+2Q\right)\,,
\]
where $Q$ is the Toeplitz operator acting on ${\mathbb C}^m$
given by \eqref{Q}. Therefore, by Theorem \ref{toeplitz},
the unitary similarity orbit of
$P_mVP_m{}_{\vert {\rm H}_m}$ is completely determined by the norms $\|f(P_mVP_m{}_{\vert {\rm H}_m})\|$, for $f\in\ct$.

The sequence $\{P_m\}_{m}$ of finite-rank projections $P_m$ converges strongly to the identity operator. Hence, because $V$ is a compact operator,
there is an $m$ such that $\|P_mVP_m-V\|<\varepsilon$.
\end{proof}

%%%%%%%%%%%%%%%
\section{Remark}
%\subsection{Scalar versus higher-dimensional phenomena}\label{h-dim}

Theorem \ref{main result} is linked to higher-dimensional
phenomena encoded by the \emph{matricial spectrum}
of $A\in\mn$ \cite{farenick1993}.

Because for every $A\in\mn$ the unital algebra $\alg A$ is
abelian, there exist unital homomorphisms $\alg A\rightarrow\mk$, for all $1\leq k \leq n-1$.  For a given $k$, let $\Hom(A, \mk)$ denote the set
of all unital homomorphisms $\alg A\rightarrow\mk$. If $\rho\in \Hom(A,\mk)$, then there is a $k$-dimensional
subspace ${\rm L}\subseteq{\mathbb C}^n$ such that
$\rho(A) \sim PAP_{\vert{\rm L}}$, where $P\in \mn$ is the unique (selfadjoint) projection with range ${\rm L}$. This subspace ${\rm L}$
is necessarily \emph{semi-invariant} under $A$; conversely, every $k$-dimensional
semi-invariant subspace of $A$ determines an element $\rho\in \Hom(A,\mk)$ \cite[Theorem 3.3.1]{Gohberg-Lancaster-Rodman-book}.

It is natural to consider the values of $\rho\in \Hom(A,\mk)$ as higher order spectra.
Specifically, consider the $k$-th matricial spectrum of $A$:
\[
\sigma_k(A)\;=\;\{\Lambda\in\mk\,:\,\Lambda=\rho(A)\mbox{ for some }\rho\in \Hom(A,\mk)\}\,.
\]
This set is closed under unitary similarity, and is itself a unitary similarity invariant of $A$. Theorem \ref{main result} is formulated in the context of
leading principal submatrices of upper triangular matrices, which in a basis-free setting means that the formulation is in the context of
invariant subspaces; if one strengthens that to semi-invariant subspaces,
then a slightly weaker hypothesis on $B$ is afforded.

\begin{proposition}\label{si}
Assume that
$A, B\in \mn$ and that $A$ is unicellular. If for each $1\leq k\leq n$ and each $\rho\in \Hom(A,\mk)$ there is a
$\varrho\in \Hom(B,\mk)$ such that
\[
\|\varrho(f(B))\|\;=\;\|\rho(f(A))\|\,,\;\mbox{for all }\,f\in \ct\,,
\]
then $B\sim A$.
\end{proposition}

\begin{proof} Without loss of generality we may assume that
$A$ and $B$ are in upper triangular form
with nonnegative entries along the superdiagonal above the main diagonal. With $A$, the entries $a_{i,i+1}$ are positive. Fix $i$ and consider
$\Lambda=PAP_{\vert{\rm L}}$ and $\Omega=PBP_{\vert{\rm L}}$,
where ${\rm L}=\mbox{Span}\,\{e_i,e_{i+1}\}$.
In this case, $\rho(X)=PXP_{\vert {\rm L}}$, for $X\in \alg A \,\cup\,\alg B$, defines an element of $\Hom(A,{\rm M}_2)$ and $\Hom(B,{\rm M}_2)$ such that
\[
\Lambda\;\sim\; \left[ \begin{array}{cc} 0 & a_{i,i+1} \\ 0&0 \end{array}\right] \quad\mbox{and}\quad
\Omega\;\sim\; \left[ \begin{array}{cc} 0 & b_{i,i+1} \\ 0&0 \end{array}\right] \,.
\]
Thus, $0\neq \|\Lambda\|=\|\Omega\|=b_{i,i+1}$ Thus, $B$ satisfies the hypothesis of Theorem \ref{main result}, which yields
our conclusion.
\end{proof}

The power of working in higher dimensions is strikingly illustrated by an important theorem of
Arveson \cite{arveson1970}: if $A,B\in \mn$ are irreducible, then $A\sim B$ if and only if
$\|A\otimes C+I\otimes D\|=\|B\otimes C+I\otimes D\|$, for all $C,D\in \mn$. This is to say that the norms
of polynomials (of degree at most $1$) in $A$, over the ring $\mn$, determine $A$ up to unitary similarity. In comparison,
Theorem \ref{main result} and Proposition \ref{si} represent a hybrid of the matricial and scalar
environments.

%%%%%%%%%%%%%%%%%%%%%%%%%%%%%%
\section*{Acknowldegement}

We thank Roger Horn and Vladimir Sergeichuk for several suggestions that have improved the presentation of these results.
We are particularly indebted to V.S. for initiating the present collaboration.
The work of the first author is supported in part by an NSERC Discovery Grant; the
second and third authors are supported in part by the
Svyatoslav Vakarchuk ``People of the Future'' Fund.

%%%%%%%%%%%%%%%%%%%%%%%%%%% bibliography %%%%%%%%%%%%%%%%%%%%%%%%%

\bibliographystyle{model1b-num-names}

\end{document}